\documentclass[a4paper, reqno, 11pt]{amsart}

\usepackage[english]{babel}
\usepackage{amsmath}
\usepackage{amssymb}
\usepackage{enumerate}
\usepackage{ifthen}
\usepackage{hyperref}
\usepackage{color}
\newcommand{\margnote}[1]{
\ifthenelse{\boolean{shownotes}}%
{\marginpar{\raggedright\tiny\texttt{#1}}}%
{}%
}

\newcommand{\hole}[1]{
\ifthenelse{\boolean{shownotes}}%
{\begin{center} \fbox{ \rule {.25cm}{0cm}
\rule[-.1cm]{0cm}{.4cm} \parbox{.85\textwidth}{\begin{center}
\texttt{#1}\end{center}} \rule {.25cm}{0cm}}\end{center}}
{}
}
\newtheorem{theorem}{Theorem}[section]

\newtheorem{lemma}[theorem]{Lemma}

\newtheorem{definition}[theorem]{Definition}

\theoremstyle{remark}

\newcommand{\e}{\varepsilon}		       
\newcommand{\R}{\mathbb{R}}

\newcommand{\T}{\mathbb{T}^2}
\newcommand{\ue}{u^{\varepsilon}}
\newcommand{\pe}{p^{\varepsilon}}
\newcommand{\be}{B^{\varepsilon}}
\newcommand{\ee}{E^{\varepsilon}}
\newcommand{\je}{j^{\varepsilon}}
\newcommand{\dive}{\mathop{\mathrm {div}}}
\newcommand{\curl}{\mathop{\mathrm {curl}}}
\newcommand{\pt}{\partial_{t}}
\newcommand{\ptt}{\partial_{tt}}
\newcommand{\uu}{u^{*}_{3}}
\newcommand{\bb}{B^{*}_{3}}
\numberwithin{equation}{section}

\begin{document}

\title[Navier-Stokes-Maxwell equations]{ Vanishing dielectric constant regime for the Navier Stokes Maxwell equations}

\author[Donatelli]{Donatella Donatelli}
\address[D. Donatelli]{\newline
Departement of Engineering Computer Science and Mathematics\\
University of L'Aquila\\
67100 L'Aquila, Italy.}
\email[]{\href{donatella.donatelli@univaq.it}{donatella.donatelli@univaq.it}}
\urladdr{\href{http://univaq.it/~donatell}{univaq.it/\~{}donatell}}

\author[Spirito]{Stefano Spirito}
\address[S. Spirito]{\newline
GSSI - Gran Sasso Science Institute \\ 67100, L'Aquila, Italy}
\email[]{\href{stefano.spirito@gssi.infn.it}{stefano.spirito@gssi.infn.it}}



\begin{abstract}
In this paper we rigorously justify the convergence of smooth solutions of the Navier-Stokes-Maxwell equations towards smooth solutions of the classical $2D$ parabolic MHD equations in the case of vanishing dielectric constant. The result is achieved by means of higher-order energy estimates.

\end{abstract}

\maketitle
\section{Introduction}
The classical Magnetohydrodynamics (MHD) equations for an electrically conducting, non magnetic, viscous incompressible fluid, e.g. plasma fluid,  in $\Omega\times(0,T)$ with $\Omega\subset\R^{d}$ ($d=2,3$) read as follows:
\begin{equation}\label{eq:MHD}
\begin{aligned}
\pt u-\Delta u+(u\cdot\nabla)u+\nabla p-(B\cdot\nabla) B&=0,\\
\pt B-\Delta B+(u\cdot\nabla )B-(B\cdot\nabla)u&=0,\\
\dive B=\dive u&=0,
\end{aligned}
\end{equation} 
where, for simplicity, we set all the physical constants equal to one.
The system \eqref{eq:MHD}, widely studied in literature and used in the applications (see \cite{ST,D}), models the evolution of the velocity $u\in\R^d$, the magnetic field $B\in\R^{d}$ and the scalar pressure $p\in\R$. Moreover, the system is accomplished with initial data, namely 
\begin{equation}\label{eq:ID}
u(x,0)=u_0(x),\,\, B(x,0)=B_0(x)\textrm{ on }\Omega\times\{t=0\},
\end{equation}
and suitable boundary conditions on $\partial\Omega\times(0,T)$.
The model \eqref{eq:MHD} is not the only system of equations used to model this kind of fluids. Another interesting model for plasma fluids is given by the Navier-Stokes-Maxwell system (see \cite{M10}):
\begin{equation}
\label{eq:mns21}
\begin{aligned}
\pt u-\Delta u+(u\cdot\nabla)u+\nabla q&=j\times B,\\ 
\pt E-\curl B&=-j,\\
\pt B+\curl E&=0,\\
\dive u&=0,\\
\dive B&=0,\\
E+(u \times B)&=j,
\end{aligned}
\end{equation}
where $E\in\R^{3}$ is the electric field and $j\in\R^{3}$ is the current density. In the case the domain $\Omega$ is two dimensional the cross products in the equations \eqref{eq:mns21} make sense by considering $u,\ B,\ E$ and $j$ with values in $\R^{3}$. The goal of this paper is to recover in a rigorous way solutions of equations \eqref{eq:MHD} from solutions of equations \eqref{eq:mns21} in a suitable limit process, that as we will see fits in the framework of singular limits. In particular, we  give a  rigorous justification of this  singular limit in the theory of magnetohydrodynamic equations. A similar limit was already considered, see for example \cite{AIM}. Before going into the mathematical details of this limiting process, in the next section we  describe the physical principles that give rise to the models we are considering.

\subsection{Singular limit and Statement of the Main Result}
The system \eqref{eq:MHD} is derived from the Navier-Stokes equations and the Maxwell equations by using the classical continuos mechanics theory and by making, as usual, smallness assumptions in order to simplify the equations taken into account. Specifically, the Maxwell equations for materials which are neither magnetic nor dielectric, are the following (see \cite{D}, \cite{EM1990}):
\begin{equation}\label{eq:Max}
\begin{aligned}
\dive E&=\frac{\rho}{\e_0} &\quad& \text{(Gauss' law)}\\
\dive B&=0 &\quad& \text{(Solenoidal nature of $B$)}\\
\curl E&=-\frac{\partial}{\partial t}B&\quad& \text{(Faraday's law in differential form)}\\
\curl B&=\mu_0\left(j+\e_0\frac{\partial}{\partial t}E\right)&\quad& \text{(Amp\`ere - Maxwell equation)},
\end{aligned}
\end{equation}
in addition we have
\begin{equation*}
\begin{aligned}
j&=\sigma(E+u\times B)&\quad& \text{(current density - Ohm's law)}\\
F&=\rho E+j\times B&\quad& \text{(electrostatic force plus Lorentz force)}.
\end{aligned}
\end{equation*}
Here $\rho$ is the {\em total charge density}, $E$ the {\em total electric field}, $B$ the {\em  magnetic field}, $\e_{0}$ the {\em electric permittivity of free space}, $\mu_{0}$ the {\em permeability of free space} and $\sigma$ the {\em conductivity}.
In MHD equations the Maxwell equations are considerably simplified. First, by assuming the  quasineutrality regime, in $F$ the contribution of the electric force $\rho E$  is small compared with the Lorentz force and then $F$ could be assumed being equal only to $j\times B$. Apparently, $\rho$ plays a role only in the Gauss' law, then we simply drop it. At this point we are left with the following form of the  Maxwell equations 
\begin{equation}
\label{eq:Max2}  
\begin{aligned}
\dive B&=0\\
\curl E&=-\frac{\partial}{\partial t}B\\
\curl B&=\mu_0\left(j+\e_0\frac{\partial}{\partial t}E\right)
\end{aligned}
\end{equation}
and the relations 
\begin{equation*}
\begin{aligned}
j&=\sigma(E+u\times B)\\
F&=j\times B.
\end{aligned}
\end{equation*}

If we  set $\sigma=1$, by using \eqref{eq:Max2} we derive the Navier-Stokes-Maxwell system:
\begin{equation}
\label{eq:mns1}
\begin{aligned}
\pt u-\Delta u+(u\cdot\nabla)u+\nabla q&=j\times B \\ 
\mu_0\e_0\pt E-\curl B&=-\mu_0j\\
\pt B+\curl E&=0\\
\dive u&=0\\
\dive B&=0\\
E+(u\times B)&=j.
\end{aligned}
\end{equation}
The last assumption in the MHD regime  is that the displacement of the currents $\mu_0\e_{0}\partial E/\partial t$ is negligible. 
Indeed  in  a typical conductor the characteristic velocity  is much smaller than the speed of the light then, the displacement of the currents can be considered small. This can be seen more clearly with a simple scaling argument. In order to get a somewhat deeper insight into the structure of possible solutions, we can identify characteristic values of relevant physical quantities: the reference time $t_{ref}$, the reference length $L_{ref}$,  the reference velocity $u_{ref}$, and the characteristic values of other composed quantities $q_{ref}$, $B_{ref}$, $E_{ref}$, $j_{ref}$. Introducing new independent and dependent variables $X' = X/X_{ref}$, omitting the primes in the resulting equations and recalling that $\mu_0\e_0=c^{-2}$, where $c$ is the speed of light, we get the following dimensionless form of the Amp\'ere - Maxwell equation
 \begin{equation}\label{eq:ampere}
\left(\frac{u_{ref}}{c}\right)^2\partial_t E-\curl B=-\bar{\nu}j
\end{equation} 
with $\bar{\nu}$ being a dimensionless constant. Then, the displacement of the current is negligible because the characteristic velocity of the fluid is much smaller than the velocity of the light. Setting $\e=\left(\frac{u_{ref}}{c}\right)^2$ we have the following $\e$-dependent dimensionless version of the Navier-Stokes-Mawell system \begin{equation}
\label{eq:mns}
\begin{aligned}
\pt\ue-\Delta\ue+(\ue\cdot\nabla)\ue+\nabla q^{\e}&=\je\times\be\\ 
\e\pt\ee-\curl\be&=-\je\\
\pt\be+\curl\ee&=0\\
\dive\ue&=0\\
\dive\be&=0\\
\ee+(\ue\times\be)&=\je
\end{aligned}
\end{equation}
supplemented with the following initial data
\begin{equation}
\ue(x,0)=\ue_0(x) \quad \be(x,0)=\be_{0}(x) \quad \ee(x,0)=\ee_{0}(x).
\label{eq:id}
\end{equation}
At a formal level we can see, that as $\e\to 0$ we have that the Amp\`ere - Maxwell equation reduces to the  Amp\`ere's  law
 \begin{equation}\label{eq:preamp}
\curl B=j
\end{equation}
Then, if we combine Ohm's law, Amp\`ere's  law with the Faraday's law  we get the following equations for the magnetic field  
\begin{equation}\label{eq:ind}
\pt B-\curl\curl B-\curl(u\times B)=0.
\end{equation}
and, concerning the equations for the velocity field ,by using \eqref{eq:preamp} we get 
\begin{equation}\label{eq:velform}
\pt u-\Delta u+(u\cdot\nabla)u+\nabla p=\curl B\times B.
\end{equation}
Then, by classical vector identities \eqref{eq:ind} is exactly the equations for the magnetic field in \eqref{eq:MHD} and, up to redefine the pressure, \eqref{eq:velform} is the equations for $u$ in \eqref{eq:MHD}.\\

 In this paper we rigorously justify the above formal limit in the case of $\Omega$ being the two dimensional torus. Our main theorem can be stated as follows.
\begin{theorem}\label{teo:main}
Let $\Omega=\T$  and $T>0$, $s>3$. Let $(u_0, B_0)\in H^{s}(\T;\R^2)$ be divergence-free vector field . Let $(u,B)\in C([0,T);H^{s}(\T;\R^2))$ be the unique smooth solution of the Cauchy problem \eqref{eq:MHD}-\eqref{eq:ID}. Then, there exist $\bar{\e}>0$ and $\ue_0$, $\be_0$ and  $\ee_0$ in $H^{s}(\T;\R^3)$ such that for any $\e<\bar{\e}$ the unique smooth solutions $\ue$, $\be$ and $\ee$ of \eqref{eq:mns}-\eqref{eq:id} satisfy:  
\begin{equation}\label{eq:convergecemain}
\begin{aligned}
&\ue\rightarrow u\textrm{ weakly$^{*}$ in }C([0,T); H^{1}(\T;\R^3)),\\
&\be\rightarrow B\textrm{ weakly$^{*}$ in }C([0,T); H^{1}(\T,\R^3)),
\end{aligned}
\end{equation}
where $u$ and $B$ are considered as three dimensional vector with vanishing third component. 
\end{theorem}

\subsection{Different interpretations of the limit}
This type of limit may have different interpretations according to the different  approaches. 
In particular it may be considered also in the context of the hydrodynamical limits of Vlasov-Maxwell equations or  in the framework of hyperbolic to parabolic relaxation theory. In fact in the paper \cite{JM12}, the authors perform a formal analysis for the hydrodynamical limit from a two- species Vlasov-Maxwell- Boltzmann equations in the regime of $\e_0$ small. In particular they consider the following form of  the scaled Vlasov-Maxwell Boltzmann system describing the dynamics of charged dilute particles,
\begin{equation}
\label{VB1}
\begin{aligned}
&\e\partial _{t}F^{\e}+v\cdot\nabla_{x}F^{\e}+(\e E^{\e}+v\times B^{\e})\cdot\nabla_{v}G^{\e}=\frac{1}{\e} Q(F^{\e}, F^{\e}),\\
&\e\partial _{t}G^{\e}+v\cdot\nabla_{x}G^{\e}+\left(\frac{E^{\e}}{\e}+\frac{v\times B^{\e}}{\e}\right)\cdot\nabla_{v}F^{\e}=\frac{1}{\e} Q(G^{\e}, F^{\e}),\\
&\e\partial_{t}E^{\e}-\nabla\times B^{\e}=-\int_{\R^{3}}vG^{\e} dv,\quad \nabla\cdot B^{\e}=0,\\
&\partial_{t}B^{\e}-\nabla\times E^{\e}, \quad \nabla\cdot E^{\e}=\frac{1}{\e}\int_{\R^{3}}G^{\e} dv,
\end{aligned}
\end{equation}
where $\e=\e_{0}$, $x$ is the position, $v$ the velocity, $F^{\e}$ is the total mass density, $G^{\e}$ the total charge density, $(E^{\e}, B^{\e})$ the electromagnetic field. Formally, as $\e\to 0$ one can recover the system \eqref{eq:MHD}, for details see Theorem 3.2 in  \cite{JM12}.
Finally, we want  to remark that the previous limit is also interesting from the point of view of the hyperbolic-parabolic relaxation limit since the system \eqref{eq:mns1} can be seen as the relaxed version of the system \eqref{eq:MHD}. In fact, let us consider the following system

\begin{equation}
\label{eq:mns2}
\begin{aligned}
\pt u-\Delta u+(u\cdot\nabla)u+\nabla q&=j\times B\\ 
\pt E-\curl B&=-j\\
\pt B+\curl E&=0\\
\dive u&=0\\
\dive B&=0\\
E+(u \times B)&=j.
\end{aligned}
\end{equation}

We perform now, the following diffusive scaling, namely for any $\e>0$ we set
\begin{equation}
\begin{aligned}
\ue(x,t)=\frac{1}{\sqrt{\e}}u\left(\frac{x}{\sqrt{\e}},\frac{t}{\e}\right) &\qquad \be=\frac{1}{\sqrt{\e}}B\left(\frac{x}{\sqrt{\e}},\frac{t}{\e}\right),\\
\ee=\frac{1}{\e}E\left(\frac{x}{\sqrt{\e}},\frac{t}{\e}\right) \qquad \je&=\frac{1}{\e}j\left(\frac{x}{\sqrt{\e}},\frac{t}{\e}\right) \qquad  q^{\e}=\frac{1}{\e}q\left(\frac{x}{\sqrt{\e}},\frac{t}{\e}\right).
\end{aligned}
\label{eq:mns3}
\end{equation}
With the previous scaling the system \eqref{eq:mns2} assumes the form \eqref{eq:mns1} and, as $\e\to 0$, at a  formal level we get the MHD equations. Let us recall that the diffusive scaling \eqref{eq:mns3} has been widely investigated in the analysis of hyperbolic-parabolic relaxation limits for weak solutions of an hyperbolic system with strongly diffusive terms, see \cite{MR00}, \cite{DM10}, \cite{DL09}, \cite{BNP04}.
For a general overview of the theory of the singular limits see the survey \cite{DM02}  and the paper \cite{DM04}, where the theory is completely set up.

\subsection{Final Remarks and Plan of the paper}
 We want to conclude this  Introduction by making some comments and pointing out some open questions. 
\begin{itemize}
\item The regularity of the initial data can be clearly relaxed. 
\item An extension of this result in the whole space should be only technical. However in the case of a bounded domain with no-slip boundary conditions the  proof of Theorem \ref{teo:main} does not work.
\item It could be possible to obtain a rate of convergence for the $(\ue, \be)$ by using a modulated energy argument as in \cite{BNP04}. 
\item A very interesting problem would be the convergence in the topology of the initial data globally in time in two dimension and locally in time in three dimension. 
\item Concerning the three dimensional case, we strongly believe that this type of limit works in the case of  small initial data for the \eqref{eq:MHD}.
\item A very interesting open problem is the convergence on three dimension in the energy space. 
\end{itemize}

Finally, we mention that similar singular limits have been considered in three space dimension, in the framework of  compressible magnetohydrodynamic equations  under the assumption of well prepared  initial data and  smooth solutions of the target system   by employing classical nonlinear energy method, see  \cite{JL2012}, \cite{KS1986i}, \cite{KS1986ii}.

The plan of the paper is  as follows. In Section 2 we collect all the definitions and the technical results we are going to use through the paper. In Section 3 we recover the a priori estimates necessary to prove our main result Theorem \ref{teo:main}. Finally, Section 4 is devoted to the proof of the Theorem \ref{teo:main}.
\section{Preliminares}
We briefly fix the notation, which is typical of space-periodic
problems. In the sequel we shall use the customary Lebesgue spaces
$L^p(\Omega)$ and Sobolev spaces $W^{k,p}(\Omega)$ and
$H^s(\Omega):=W^{s,2}(\Omega)$, with $\Omega := ]0,2\pi[^2$; for
simplicity we shall do not distinguish between scalar and vector
valued functions. Since we shall work with periodic boundary
conditions the spaces are made of periodic functions and in the
Hilbertian case $p=2$ we can easily characterize them by using Fourier Series on the 2D torus. 
We use
$\|\cdot\|_p$ to denote the $L^{p}(\T)$ norm and we impose the zero
mean condition and on velocity, the pressure and the magnetic field. We will denote by $H^{s}(\T)$, $s=1, 2$, the classical Sobolev spaces. Moreover, $L^p(0,T;X)$ denotes  the classical Bochner spaces endowed with
the norm
\begin{equation*}
  \|f\|_{L^p(0,T;X)}:=
  \left\{\begin{aligned}
      &\left(\int_0^T\|f(t)\|_X\right)^{1/p}\quad &\text{if }1\leq p<\infty,
      \\
      &\sup_{0\leq t\leq T}\|f(t)\|_X &\text{if }p=+\infty,
\end{aligned}
\right.
\end{equation*}
Since we assumed divergence-free condition and zero average for $u$ and $B$ on $\T$ the following  norm equivalences hold,
\begin{equation*}
\begin{aligned}
&\|u\|_{H^{2}}\cong\|\Delta u\|_2,\qquad&\|u\|_{H^{1}}\cong\|\nabla u\|_{2},\\
&\|B\|_{H^{2}}\cong\|\Delta B\|_2,\qquad&\|B\|_{H^{1}}\cong\|\nabla B\|_{2}.
\end{aligned}
\end{equation*} 
We will use also the following standard inequalities:
\begin{itemize}
\item The Gagliardo-Nirenberg inequality, namely
\begin{equation}\label{eq:GN}
\|f\|_{p}\leq C\|\nabla f\|_{r}^{\alpha}\|f\|_{q}^{1-\alpha},
\end{equation}
where
\begin{equation*}
\frac{1}{p}=\left(\frac{1}{r}-\frac{1}{2}\right)\alpha+\frac{1-\alpha}{q}
\end{equation*}
and $\alpha\in[0,1]$.
\item The Kato-Ponce inequality, namely 
\begin{equation}\label{eq:KP}
\|fg\|_{H^s}\leq C(\|f\|_{\infty}\|g\|_{H^s}+\|g\|_{\infty}\|f\|_{H^s})
\end{equation}
which holds for any $s>0$. 
\item The Brezis-Gallouet inequality (see Lemma 2 in \cite{BG})
\begin{equation}\label{eq:BG}
\|f\|_{\infty}\leq C\|f\|_{H^{1}}(1+(\ln(1+\|f\|_{H^{2}}))^{\frac{1}{2}})
\end{equation}
which holds for any $u\in H^{2}$. 
\end{itemize}
Now, we recall some important results concerning the equations \eqref{eq:MHD}. Let us start with the definition of weak solutions for the Cauchy problem \eqref{eq:MHD}-\eqref{eq:ID}. 
\begin{definition}\label{def:ws}
A pair $(u,B)$ is a weak solution of the Cauchy problem \eqref{eq:MHD}-\eqref{eq:ID} if 
\begin{equation*}
u, B\in C([0,T);L^2_{weak}(\T;\R^2))\cap L^{\infty}((0,T);L^{2}(\T;\R^2)\cap L^{2}((0,T);H^{1}(\T;\R^2))\\
\end{equation*}
and the equations \eqref{eq:MHD} are satisfied in the sense of distribution for any divergence-free test function belonging to the space $C_{c}^{\infty}([0,T);C^{\infty}_{per}(\T;\R^2))$. 
\end{definition}
The following global regularity and uniqueness theorem has been proved in \cite{ST}. 
\begin{theorem}\label{teo:exMHD}
Let $s>3$ and $u_0, B_0\in H^{s}(\T;\R^2)$.  There exists a unique global smooth solution $(u,B)$ of the Cauchy problem \eqref{eq:MHD}-\eqref{eq:ID} 
 such that:
\begin{equation*}
\begin{aligned}
u\in C([0,T); H^{s}(\T;\R^2)),\\
B\in C([0,T); H^{s}(\T;\R^2)).
\end{aligned}
\end{equation*}
Moreover, $(u,B)$ is also unique in the class of weak solutions in the sense of Definition \ref{def:ws}.
\end{theorem}  
Concerning the Navier-Stokes-Maxwell system the global existence of smooth solutions has been proved in \cite{M10}. 
\begin{theorem}\label{teo:exmsn}
Let $s>3$ and $u^{\e}_0$, $B^{\e}_0$and  $E_0^{\e}$ be in $H^{s}(\T;\R^3)$, with $u^{\e}_0$ and $B_0^{\e}$ divergence-free.
Let $\e>0$ fixed and arbitrary.Then, there exists a unique global smooth solution $(\ue,\be,\ee)$ of the Cauchy problem \eqref{eq:mns}-\eqref{eq:id} with 
\begin{equation*}
\begin{aligned}
\ue\in C([0,T); H^{s}(\T;\R^3)),\\
\be\in C([0,T); H^{s}(\T;\R^3)),\\
\ee\in C([0,T);H^{s}(\T;\R^3)).
\end{aligned}
\end{equation*}
\end{theorem}
This result has been extended to the three-dimensional space with small initial data in \cite{GIM}. We want to point out that the global existence of weak solutions {\em a l\`a} Leray-Hopf is an open problem even in two dimensions, see \cite{M10}.  

\section{A priori estimates}\label{sec:est}

In this section we will recover the main a priori estimates necessary to prove Theorem \ref{teo:main}. Let $\ue_0$, $\be_0$ and $\ee_0$ be  smooth initial data and $\ue$, $\be$ and $\ee$ the unique global smooth solutions of the Cauchy problem \eqref{eq:mns}-\eqref{eq:id}. The first basic $\e$-independent a priori estimate for the system \eqref{eq:mns} is the classical energy estimate, see \cite{M10}. 
\begin{lemma}
\label{lem:est1}
Let $(\ue,\be, \ee)$ be a solution of the system \eqref{eq:mns}, then the following differential equality holds. 
\begin{equation}
\label{eq:est1}
\frac{d}{dt}\left(\int|\ue|^2+|\be|^2+\e|\ee|^2\right)+2\int|\nabla\ue|^2+|\je|^2=0.
\end{equation}
\end{lemma}
\begin{proof}
The proof is rather standard. We multiply the  first three equations of \eqref{eq:mns} by  $\ue$,   $\be$ and  $\ee$ respectively. We integrate by parts in space, by using the definition of $\je$ and adding up everything we  obtain \eqref{eq:est1}. 
\end{proof}
The a priori estimates of Lemma \ref{lem:est1} are clearly not enough to justify the limit as $\e$ goes to zero. In order to simplify the computations to get further a priori estimates, we  rewrite  the system \eqref{eq:mns} in the following form,
\begin{equation}
\label{eq:main}
\begin{aligned}
\pt\ue-\Delta\ue+(\ue\cdot\nabla)\ue+\nabla\pe& =(\be\cdot\nabla)\be-\e\pt\ee\times\be,\\
\e\ptt\be+\partial_t\be-\Delta\be+(\ue\cdot\nabla)\be& =(\be\cdot\nabla)\ue,\\
\e\pt\ee+\ee-\curl\be &=-(\ue\times\be),\\
\dive\ue &=0,\\
\dive\be &=0.
\end{aligned}
\end{equation}
The initial data for the system \eqref{eq:main} are 
\begin{equation*}
\begin{aligned}
\ue(x,0)&=\ue_0(x),\\
\be(x,0)&=\be_0(x),\\
\pt\be(x,0)&=\curl{E_0^{\e}}(x),\\
\ee(x,0)&=\ee_0(x).
\end{aligned}
\end{equation*} 
Note that the value of $\partial_t\be$ at time $t=0$ is obtained from the system \eqref{eq:mns} and the pressure has been redefined.
The next Lemma is the first main a priori estimate of the paper. Before stating it we define the following quantities 
\begin{align}
\label{eq:e1}
\mathcal{E}_1(t)&=\int\frac{|\ue|^2}{2}+\frac{|\be+2\e\pt\be|^2}{2}+3\e|\nabla\be|^2+\e^2|\pt\be|^2+\e\frac{|\ee|^2}{2}\\
\label{eq:2}
\mathcal{D}_1(t)&=\int\e^2|\pt\ee|^2+\frac{1}{2}|\nabla\ue|^2+\frac{1}{2}|\nabla\be|^2+\e|\pt\be|^2.
\end{align}

\begin{lemma}\label{lem:e1}
Let $(\ue,\be,\ee)$ be a smooth solution of \eqref{eq:mns}-\eqref{eq:id} in $\T\times(0,T)$. There exists an absolute constant $C_1>0$ such that, if 
\begin{equation}\label{eq:cond1}
\|\ue(t,\cdot)\|_{\infty}+\|\be(t,\cdot)\|_{\infty}\leq\frac{C_1}{\sqrt{\e}}\quad \text{ for any $t\in[0,T)$}
\end{equation}
then, 
\begin{equation}\label{eq:est2}
\frac{d}{dt}\mathcal{E}_1(t)+\mathcal{D}_1(t)\leq 0\qquad\textrm{for any $t\in(0,T)$}.
\end{equation}
\end{lemma}
\begin{proof}
We multiply the first equation in \eqref{eq:main} by $\ue$, after integration by parts we get
\begin{equation}\label{eq:A}
\frac{d}{dt}\int\frac{|\ue|^2}{2}+\int|\nabla\ue|^2=\int(\be\cdot\nabla)\be\cdot\ue-\int(\e\pt\ee\times\be)\cdot\ue.
\end{equation}
Then, we consider the second equation of \eqref{eq:main} rewritten as follows 
\begin{equation}\label{eq:est21}
2\e\ptt\be+\partial_t\be-\Delta\be+(\ue\cdot\nabla)\be -\e\ptt\be=(\be\cdot\nabla)\ue.
\end{equation}
We multiply \eqref{eq:est21} by $\be+6\e\pt\be$, and after integrating by parts we get 
\begin{equation*}
\begin{aligned}
&\frac{d}{dt}\left(\int\frac{|\be+2\e\pt\be|^2}{2}+3\e|\nabla\be|^2+\e^2|\pt\be|^2\right)+4\e\int|\pt\be|^2\\
&+\int|\nabla\be|^2-\e\int\ptt\be\cdot\be+6\e\int\pt\be\cdot\curl(\be\times\ue)\\
&=\int(\be\cdot\nabla)\ue\cdot\be,
\end{aligned}
\end{equation*}
which can be reformulated as follows,
\begin{equation}
\label{eq:B}
\begin{aligned}
&\frac{d}{dt}\left(\int\frac{|\be+2\e\pt\be|^2}{2}+3\e|\nabla\be|^2+\e^2|\pt\be|^2\right)+\int|\nabla\be|^2+\e\int|\pt\be|^2\\
&+3\e\int|\pt\be+\curl(\be\times\ue)|^2-\e\int\ptt\be\cdot\be\\
&-3\e\int|\curl(\be\times\ue)|^2=\int(\be\cdot\nabla)\ue\cdot\be.
\end{aligned}
\end{equation}
Finally, we multiply the third equation of \eqref{eq:main} by $\e\pt\ee$ and, after an integration by parts we have
\begin{equation}
\label{eq:est22}
\frac{d}{dt}\e\!\int\frac{|\ee|^2}{2}+\e^2\!\!\int|\pt\ee|^2-\int\be\cdot\e\pt\curl\ee=-\int(\ue\times\be)\cdot\e\pt\ee.
\end{equation}
By using  $\eqref{eq:mns}_{3}$ and the following  standard property of vector and scale products 
\begin{equation*}
-(\ue\times\be)\cdot\e\pt\ee=\ue\cdot(\e\pt\ee\times\be),
\end{equation*}
 \eqref{eq:est22} becomes
\begin{equation}
\label{eq:C}
\frac{d}{dt}\int\e\frac{|\ee|^2}{2}+\e^2\int|\pt\ee|^2+\e\int\ptt\be\cdot\be=\int\ue\cdot(\e\pt\ee\times\be).
\end{equation}
By adding up \eqref{eq:A}, \eqref{eq:B} and \eqref{eq:C} we get 
\begin{equation}
\label{eq:est23}
\begin{aligned}
&\frac{d}{dt}\mathcal{E}_1(t)+\int|\nabla\be|^2+\int|\nabla\ue|^2+\e^2\int|\pt\ee|^2+\e\int\ptt\be\cdot\be\\
&+\e\int|\pt\be|^2+3\e\int|\pt\be+\curl(\be\times\ue)|^2-\e\int\ptt\be\cdot\be\\
&-3\e\int|\curl(\be\times\ue)|^2=0.
\end{aligned}
\end{equation}
At this point we treat the term with negative sign in the right hand side. We have that 
\begin{equation*}
\begin{aligned}
&\int|\curl(\be\times\ue)|^2\leq \int|(\ue\cdot\nabla)\be|^2+|(\be\cdot\nabla)\ue|^2\\
&\leq C(\|\ue\|_{\infty}^2+\|\be\|_{\infty}^2)\left(\frac{1}{2}\|\nabla\ue\|_2^2+\frac{1}{2}\|\nabla\be\|_2^2\right),
\end{aligned}
\end{equation*}
where $C>0$ is an absolute constant. Then \eqref{eq:est23} becomes an inequality and we get \eqref{eq:est2} with $C_{1}=\sqrt{\frac{1}{3C}}$.
\end{proof}
We need also higher order a priori estimates independent on $\e$. This will be done in  the next Lemma. Let us define the following quantities
 \begin{equation*}
\begin{aligned}
\mathcal{E}_2(t)&=\int\frac{|\nabla\ue|^2}{2}+\e\frac{|\Delta\ue|^2}{2}+\frac{|\nabla\be+2\e\pt\nabla\be|^{2}}{2}\\
&+\int3\e|\Delta\be|^2+\e^2|\pt\nabla\be|^2+\e\frac{|\nabla\ee|^2}{2}.\\
\mathcal{D}_2(t)&=\frac{1}{4}\left(\int|\Delta\ue|^2+|\Delta\be|^2+\e|\pt\nabla\ue|^2+\e^2|\pt\nabla\ee|^2\right).
\end{aligned}
\end{equation*}
Then, the following lemma holds.
\begin{lemma}\label{lem:e2}
 Let $(\ue,\be,\ee)$ be a smooth solution of \eqref{eq:mns}-\eqref{eq:id} in $\T\times(0,T)$. There exists an absolute constant $C_2>0$ such that if
\begin{equation}\label{eq:cond2}
\|\ue(t,\cdot)\|_{\infty}+\|\be(t,\cdot)\|_{\infty}\leq\frac{C_2}{\sqrt{\e}}\quad \text{for any $t\in[0,T)$}
\end{equation}
 then, the following differential inequality holds,
\begin{equation}
\label{eq:est3}
\frac{d}{dt}\mathcal{E}_2(t)+\mathcal{D}_2(t)\leq C(1+\mathcal{E}_1(t))\mathcal{D}_1(t)\mathcal{E}_2(t).
\end{equation}  
\end{lemma}
\begin{proof}
We start by multiplying the first equation of \eqref{eq:main} by $-\Delta\ue$, after an integration by parts we get 
\begin{equation}
\label{eq:A1}
\begin{aligned}
\frac{d}{dt}\int\frac{|\nabla\ue|^2}{2}+\int|\Delta\ue|^2&=\int\ue\cdot\nabla\ue\cdot\Delta\ue-\int\be\cdot\nabla\be\Delta\ue\\
                                                                                   &+\int(\e\pt\ee\times\be)\cdot\Delta\ue.
\end{aligned}
\end{equation}
The second estimate we perform is obtained by multiplying the first equation of \eqref{eq:main} by $-\e\Delta\pt\ue$
\begin{equation}
\label{eq:B1}
\begin{aligned}
\frac{d}{dt}\int\e\frac{|\Delta\ue|^2}{2}+\e\int|\nabla\pt\ue|^2&=\e\int\ue\cdot\nabla\ue\cdot\Delta\pt\ue\\
                                                                                             &+\e\int(\e\pt\ee\times\be)\Delta\pt\ue\\
                                                                                             &-\e\int\be\nabla\be\Delta\pt\ue.
\end{aligned}
\end{equation}

Then, we multiply \eqref{eq:est21} by $-\Delta(\be+6\e\pt\be)$ and we get 
\begin{equation}
\label{eq:C1}
\begin{aligned}
&\frac{d}{dt}\left(\int\frac{|\nabla\be+2\e\pt\nabla\be|^2}{2}+3\e|\Delta\be|^2+\e^2|\pt\nabla\be|^2\right)+\int|\Delta\be|^2\\
&+\e\int|2\pt\nabla\be+\frac{3}{2}\nabla\curl(\be\times\ue)|^2-\frac{9}{4}\e\int|\nabla\curl(\be\times\ue)|^2\\
&+\e\int\ptt\be\cdot\Delta\be=-\int\curl(\ue\times\be)\Delta\be.
\end{aligned}
\end{equation}
Finally, we multiply the third equation of  \eqref{eq:main} by $-\e\pt\Delta\ee$ and we obtain
\begin{equation}\label{eq:est31}
\begin{aligned}
&\frac{d}{dt}\int\e\frac{|\nabla\ee|^2}{2}+\e^2\int|\pt\nabla\ee|^2+\int\curl\be\e\pt\Delta\ee\\
&=\int(\ue\times\be)\e\pt\Delta\ee
\end{aligned}
\end{equation}
Concerning the third term of the left-hand  side of \eqref{eq:est31} by using again $\eqref{eq:mns}_{3}$ we have
\begin{equation*}
\begin{aligned}
\e\int\curl\be\pt\Delta\ee&=\e\int\be\pt\Delta\curl\ee=-\e\int\be\pt\Delta\pt\be\\
                                      &=-\e\int\be\ptt\Delta\be=-\e\int\Delta\be\ptt\be.
\end{aligned}
\end{equation*}
Then \eqref{eq:est31} becomes 
\begin{equation}
\label{eq:D1}
\begin{aligned}
&\frac{d}{dt}\int\e|\nabla\ee|^2+\e^2\int|\pt\nabla\ee|^2-\e\int\Delta\be\ptt\be\\
&=\int(\ue\times\be)\e\pt\Delta\ee.
\end{aligned}
\end{equation}
By summing up \eqref{eq:A1}, \eqref{eq:B1}, \eqref{eq:C1} and \eqref{eq:D1} we get 
\begin{equation}\label{eq:est32}
\begin{aligned}
&\frac{d}{dt}\mathcal{E}_2(t)+\int|\Delta\ue|^2+\e\int|\nabla\pt\ue|^2+\int|\Delta\be|^2\\
&+\e^2\int|\pt\nabla\ee|^2+\e\int|2\pt\nabla\be+\frac{3}{2}\nabla\curl(\be\times\ue)|^2\\
&-\frac{9}{4}\e\int|\nabla\curl(\ue\times\be)|^2\leq(I)+(II)+(III)+(IV).
\end{aligned}
\end{equation}
Where the terms on the right-hand side are respectively 
\begin{equation*}
(I)=\left|\int(\ue\cdot\nabla)\ue\Delta\ue-(\be\cdot\nabla)\be\Delta\ue+(\ue\cdot\nabla)\be\Delta\be-(\be\cdot\nabla)\ue\Delta\be\right|,
\end{equation*}
\begin{equation*}
(II)=\left|\e\int(\ue\cdot\nabla)\ue\Delta\pt\ue\right|,
\end{equation*}
\begin{equation*}
(III)=\left|\e\int(\be\cdot\nabla)\be\Delta\pt\ue\right|,
\end{equation*}
\begin{equation*}
(IV)=\left|\int(\e\pt\ee\times\be)\Delta\ue+\e\int(\ue\times\be)\pt\Delta\ee+ \int(\e\pt\ee\times\be)\e\Delta\pt\ue\right|.
\end{equation*}
We estimate all the previous termst separately. By integrating by parts we have that 

\begin{align}
\nonumber(I)&\leq C\int|\nabla \ue|^3+|\nabla \be|^2|\nabla\ue|\\
\nonumber   &\leq \|\nabla\ue\|_3^3+\|\nabla\be\|_4^2\|\nabla\ue\|_2\\
\nonumber   &\leq C\|\nabla\ue\|_2^2\|\Delta\ue\|_2+\|\nabla\be\|_2\|\nabla\ue\|_2\|\Delta\be\|_2\\
\label{eq:1}  &\leq C(\|\nabla\ue\|_2^2+\|\nabla\be\|_2^2)\|\nabla\ue\|_2^2+\frac{1}{32}\|\Delta\ue\|_2^2+\frac{1}{32}\|\Delta\be\|_2^2.
\end{align}
Where we have used the Gagliardo-Nirenberg inequality \eqref{eq:GN} first with $p=3$ and then with $p=4$ and Young inequality. Next we estimate the terms $(II)$ and $(III)$ for which we simply use Young inequality,
\begin{equation}\label{eq:2}
(II)\leq C\e\int(|\nabla((\ue\cdot\nabla)\ue))|^2+\frac{\e}{32}\|\nabla\pt\ue\|_2^2,
\end{equation}
\begin{equation}\label{eq:3}
(III)\leq C\e\int|\nabla((\be\cdot\nabla)\be)|^2+\frac{\e}{32}\|\nabla\pt\ue\|_2^2.
\end{equation}
The term $(IV)$ is a little bit troublesome. We split $(IV)$ into two parts, $(IV)_{1}$ and $(IV)_{2}$. First we consider $(IV)_1$ defined as follows 
\begin{equation*}
(IV)_1=\int(\e\pt\ee\times\be)\Delta\ue+(\ue\times\be)\e\pt\Delta\ee
\end{equation*}
By integrating by parts the second term in $(IV)_{1}$ we get 
\begin{equation*}
(IV)_1=\int(\e\pt\ee\times\be)\partial_{kk}\ue-\int(\partial_k\ue\times\be)\e\pt\partial_k\ee-\int(\ue\times\partial_k\be)\e\pt\partial_k\ee.
\end{equation*}
We integrate again by parts only the second term in $(IV)_{1}$, then
\begin{equation*}
\begin{aligned}
(IV)_1&=\int(\e\pt\ee\times\be)\partial_{kk}\ue+\int(\partial_{kk}\ue\times\be)\e\pt\ee\\
&+\int(\partial_k\ue\times\partial_k\be)\e\pt\ee-\int(\ue\times\partial_k\be)\e\pt\partial_k\ee\\
&=\int(\e\pt\ee\times\be)\partial_{kk}\ue-\int(\e\pt\ee\times\be)\partial_{kk}\ue\\
&+\int(\partial_k\ue\times\partial_k\be)\e\pt\ee-\int(\ue\times\partial_k\be)\e\pt\partial_k\ee\\
&=\int(\partial_k\ue\times\partial_k\be)\e\pt\ee-\int(\ue\times\partial_k\be)\e\pt\partial_k\ee\\
&=(IV)_{11}+(IV)_{12},
\end{aligned}
\end{equation*}
where standard vector identities have been used in the third line. Let us now estimate the term $(IV)_{11}$. 
By using H\"older inequality, Gagliardo Nirenberg inequality \eqref{eq:GN} with $p=4$ and Young inequality we have 
\begin{align}
\nonumber (IV)_{11}&\leq \e C\int|\nabla\ue||\nabla\be||\pt\ee|\\
\nonumber             &\leq \e C\|\nabla\ue\|_4\|\nabla\be\|_4\|\pt\ee\|_2\\
\nonumber             &\leq C\e\|\nabla\ue\|_{2}^{\frac{1}{2}}\|\nabla\be\|_{2}^{\frac{1}{2}}\|\Delta\ue\|_{2}^{\frac{1}{2}}\|\Delta\be\|_{2}^{\frac{1}{2}}\|\pt\ee\|_{2}\\
\nonumber             &\leq C\sqrt{\lambda}\e^2\|\nabla\ue\|_2\|\nabla\be\|_2\|\pt\ee\|_2^2+\frac{1}{\sqrt{\lambda}}\|\Delta\ue\|_2\|\Delta\be\|_2\\
\nonumber             &\leq C\sqrt{\lambda}\e^2\|\pt\ee\|_2^2(\|\nabla\ue\|_2^2+\|\nabla\be\|_2^2)+\frac{h}{2\lambda}\|\Delta\ue\|_{2}^2+\frac{1}{2h}\|\Delta\be\|_2^2\\
\nonumber             &\leq C\sqrt{\lambda}\e^2\|\pt\ee\|_2^2(\|\nabla\ue\|_2^2+\|\nabla\be+2\e\pt\nabla\be\|_2^2+4\e^2\|\pt\nabla\be\|_2^2)\\
\nonumber             &+\frac{h}{2\lambda}\|\Delta\ue\|_{2}^2+\frac{1}{2h}\|\Delta\be\|_2^2
\end{align}

and we conclude by choosing $h$ and $\lambda$ such that 
\begin{align}
\nonumber (IV)_{11}&\leq C \e^2\|\pt\ee\|_2^2(\|\nabla\ue\|_2^2+\|\nabla\be+2\e\pt\nabla\be\|_2^2+4\e^2\|\pt\nabla\be\|_2^2)\\
\label{eq:411}            &+\frac{1}{32}\|\Delta\ue\|_{2}^2+\frac{1}{32}\|\Delta\be\|_2^2.
\end{align}
Next we estimate the term $(IV)_{12}$, by using again H\"older inequality, Gagliardo Nirenberg inequality \eqref{eq:GN} with $p=4$ and Young inequality we have
\begin{align}
\nonumber(IV)_{12}&\leq C\e\int|\ue||\nabla\be||\pt\nabla\ee|\\
\nonumber              &\leq C\|\ue\|_{4}^2\|\nabla\be\|_{4}^2+\frac{1}{32}\e^2\|\pt\nabla\ee\|_{2}^{2}\\ 
\nonumber              &\leq C\|\ue\|_2\|\nabla\ue\|_2\|\nabla\be\|_2\|\Delta\be\|_2+\frac{1}{32}\e^2\|\pt\nabla\ee\|_{2}^{2}\\
\label{eq:412}             &\leq C\|\ue\|_2^2\|\nabla\ue\|_2^2\|\nabla\be\|_2^2+\frac{1}{32}\|\Delta\be\|_2^2+\frac{1}{32}\e^2\|\pt\nabla\ee\|_{2}^{2}.
\end{align}
Now, we consider the term $(IV)_{2}$. Again we integrate by parts to get 
\begin{equation*}
\begin{aligned}
(IV)_{2}&=-\int\e(\pt\partial_k\ee\times\be)\e\pt\partial_k\ue-\int\e(\pt\ee\times\partial_k\be)\e\partial_k\pt\ue\\
            &=(IV)_{21}+(IV)_{22}.
\end{aligned}
\end{equation*}
The term $(IV)_{21}$ is estimated by using H\"older and Young inequality as follows,
\begin{align}
\nonumber (IV)_{21}&\leq C \e\int\e|\pt\nabla\ee||\be||\nabla\pt\ue|\\
\nonumber               &\leq C\e\int\e^2|\pt\nabla\ee|^2|\be|^2+\frac{\e}{32}\|\nabla\pt\ue\|_2^2\\
\label{eq:421}              &\leq C\e\|\be\|_{\infty}^2\e^2\|\pt\nabla\ee\|_2^2+\frac{\e}{32}\|\nabla\pt\ue\|_2^2.
\end{align}
Finally, we consider the term $(IV)_{22}$
\begin{align}
\nonumber (IV)_{22}&\leq C \e^2\int|\pt\ee||\nabla\be||\nabla\pt\ue|\\
\nonumber              &\leq C\e\int\e^2|\pt\ee|^2|\nabla\be|^2+\frac{\e}{32}\|\nabla\pt\ue\|_2^2\\
\nonumber              &\leq C\e^3\|\nabla\be\|_4^2\|\pt\ee\|_{4}^2+\frac{\e}{32}\|\nabla\pt\ue\|_2^2\\
\nonumber              &\leq C\e^2\|\nabla\be\|_2\|\Delta\be\|_2\|\pt\ee\|_{2}\e\|\pt\nabla\ee\|_{2}+\frac{\e}{32}\|\nabla\pt\ue\|_2^2\\
\nonumber              &\leq C\e^4\|\pt\ee\|_2^2\|\nabla\be\|_2^2\|\Delta\be\|_2^2+\frac{\e^2}{32}\|\pt\nabla\ee\|_2^2+\frac{\e}{32}\|\nabla\pt\ue\|_2^2\\
\nonumber             &\leq C\e^2\|\pt\ee\|_2^2\e\|\nabla\be\|_2^2\e\|\Delta\be\|_2^2\\
\label{eq:422}        &+\frac{\e^2}{32}\|\pt\nabla\ee\|_2^2+\frac{\e}{32}\|\nabla\pt\ue\|_2^2,
\end{align}
where we have used as in the other terms H\"older inequality, Gagliardo Nirenberg inequality \eqref{eq:GN} with $p=4$ and Young inequality.
By using the estimates \eqref{eq:1}-\eqref{eq:422} in \eqref{eq:est32} and taking into account  the definition of $\mathcal{E}_2(t)$we get 
\begin{equation}\label{eq:est33}
\begin{aligned}
&\frac{d}{dt}\mathcal{E}_2(t)+\frac{\e^2}{4}\|\pt\nabla\ee\|_2^2+\frac{1}{4}\|\Delta\ue\|_2^2+\frac{\e}{4}\|\nabla\pt\ue\|_2^2+\frac{1}{4}\|\Delta\be\|_2^2\\
&+\frac{\e^2}{2}\|\pt\nabla\ee\|_2^2+\frac{1}{2}\|\Delta\ue\|_2^2+\frac{1}{2}\|\Delta\be\|_2^2-C\e\|\be\|_{\infty}^2\e^2\|\pt\nabla\ee\|_2^2\\
&-\frac{9}{4}\e\int|\nabla\curl(\ue\times\be)|^2-C\e\int(|\nabla((\ue\cdot\nabla)\ue))|^2\\&
-C\e\int|\nabla((\be\cdot\nabla))\be|^2\leq C \e^2\|\pt\ee\|_2^2\mathcal{E}_2(t)+\|\ue\|_2^2\|\nabla\be\|_2^2\mathcal{E}_2(t)\\
&C\e^2\|\pt\ee\|_2^2\e\|\nabla\be\|_2^2\mathcal{E}_2(t)+C(\|\nabla\ue\|_2^2+\|\nabla\be\|_2^2)\mathcal{E}_2(t)\\
&\leq C(1+\mathcal{E}_1(t))\mathcal{D}_1(t)\mathcal{E}_2(t).
\end{aligned}
\end{equation}
As in the previous Lemma we need to estimate the term with negative sign on the left-hand side of \eqref{eq:est33}. 
By using the Kato  inequality \eqref{eq:KP} we have that 
\begin{equation}
\label{eq:e29}
\begin{aligned}
&\frac{9}{4}\int|\nabla\curl(\ue\times\be)|^2+C\int(|\nabla(\dive(\ue\otimes\ue))|^2+C\int|\nabla(\dive(\be\otimes\be))|^2\\
&+C\|\be\|_{\infty}^2\e^2\|\pt\nabla\ee\|_2^2\\
&\leq C( \|\ue\be\|_{H^2}^2+\|\ue\ue\|_{H^2}^2+\|\be\be\|_{H^2}^2+\|\be\|_\infty^2\e^2\|\pt\nabla\ee\|_2^2)\\
&\leq C(\|\ue\|_{\infty}^2+\|\be\|_{\infty}^2)(\|\ue\|_{H^2}^2+\|\be\|_{H^2}^2+\e^2\|\pt\nabla\ee\|_2^2)\\
&\leq C^*(\|\ue\|_{\infty}^2+\|\be\|_{\infty}^2)\left(\frac{1}{2}\|\Delta\ue\|_{2}^2+\frac{1}{2}\|\Delta\be\|_{2}^2+\frac{\e^2}{2}\|\nabla\pt\ee\|_2^2\right).
\end{aligned}
\end{equation}
Then by using \eqref{eq:e29} in \eqref{eq:est33} we get \eqref{eq:est3} with $C_2=\sqrt{\frac{1}{C^*}}$.
\end{proof}
\section{Proof of the main theorem}
In this section we prove Theorem \ref{teo:main}. We divide the proof  in several steps.\\
\\
{\em Step 1. Construction on the initial data.}\\
\\
We set $C_3=\min\{C_1,C_2\}$. Let $(u_0, B_0)$ in $H^{s}(\T;\R^2)\times H^{s}(\T;\R^2)$ be the divergence free initial data for \eqref{eq:MHD}. We need to construct the initial data for the system \eqref{eq:mns}. By using a standard regularization argument, see for example \cite{M07}, we obtain two smooth sequences $\ue_0$ and $\be_0$. Moreover, by choosing $\e$ small enough we get  
\begin{equation*}
\|\ue_0\|_{\infty}\leq\frac{C_4}{\sqrt{\e}}\qquad \|\be_0\|_{\infty}\leq\frac{C_4}{{\sqrt{\e}}}.
\end{equation*}
with $C_4<C_3$. Then, we consider $\ue_0$ and $\be_0$ embedded in $\R^3$ by setting the third component to zero. 
The initial datum $\ee_0$ for the electric field will be constructed in two steps. First we solve  
\begin{equation}\label{eq:ide}
\curl E_0=-\partial_tB\big|_{t=0}
\end{equation}
endowed with periodic boundary conditions. We again consider $\pt B\big|_{t=0}$ as a three dimensional vector by setting the third component to zero and the value of $\partial_tB$ at time $t=0$ is obtained from the second equation of \eqref{eq:MHD}. Once \eqref{eq:ide} has been solved we construct $E_0^{\e}$ by a simple regularization argument.\\
\\
{\em Step 2. Global in time estimates for $(\ue, \be, \ee)$.}\\
\\
First of all we prove the uniform $L^{\infty}$ bounds for $\ue$ and  $\be$ required in Lemma \ref{lem:e1} and \ref{lem:e2}.
By Theorem \ref{teo:exmsn} there exists a unique smooth solution $(\ue, \be, \ee)$ of \eqref{eq:mns} starting from the initial data we have constructed in Step 1. Let $\delta<C_3-C_4$ and $T^{\e,\delta}=\min\{T^{\e,\delta}_{1}, T^{\e,\delta}_2\}$ where $T^{\e,\delta}_{i}$ are defined as follows:
\begin{equation*}
\begin{aligned}
T_{1}^{\e,\delta}&=\sup\left\{0\leq t\leq T;\,\sup_{0\leq\tau\leq t}\|\ue(\tau)\|_{\infty}\leq\frac{C_4+\delta}{2\sqrt{\e}}\right\},\\
T_{2}^{\e,\delta}&=\sup\left\{0\leq t\leq T;\,\sup_{0\leq\tau\leq t}\|\be(\tau)\|_{\infty}\leq\frac{C_4+\delta}{2\sqrt{\e}}\right\}.
\end{aligned}
\end{equation*}
We have that $T^{\e,\delta}>0$ because of the continuity in time with value in $H^{2}(\T)$ of $\ue$ and $\be$. 
We want to show that $T^{\e,\delta}=T$. If $T^{\e,\delta}<T$, then there exist $\alpha>0$ such that for all $t<T^{\e,\delta}+\alpha$
\begin{equation*}
\|\ue\|_{\infty}+\|\be\|_{\infty}< \frac{C_3}{\sqrt{\e}}.
\end{equation*}
Moreover, by using Lemma \ref{lem:e1} and Lemma \ref{lem:e2}, we get 
\begin{equation*}
\|\ue(T^{\e,\delta},\cdot)\|_{H^1}+\|\be(T^{\e,\delta},\cdot)\|_{H^1}+\sqrt{\e}\|\ue(T^{\e,\delta},\cdot)\|_{H^2}+\sqrt{\e}\|\be(T^{\e,\delta},\cdot)\|_{H^2}\leq C_5.
\end{equation*}
By using the definition of $T^{\e,\delta}$ and the Brezis-Gallouet inequality \eqref{eq:BG}, we have:
\begin{equation}
\label{eq:fin}
\begin{aligned}
\frac{C_{4}+\delta}{\sqrt{\e}}&=\|\ue(T^{\e,\delta})\|_{\infty}+\|\be(T^{\e,\delta})\|_{\infty}\\
                                           &\leq C\|\be\|_{H^1}(1+|\ln\|\ue\|_{H^{2}}|)+C\|\ue\|_{H^1}(1+|\ln\|\be\|_{H^{2}}|)\\
                                           &\leq C_6(1+|\ln\e|).
\end{aligned}
\end{equation}
where $C_6$ depends only on the initial data.  
Note that $\delta$ is a fixed number depending only on the constants $C_3$ and $C_4$. Then, there exists $\bar{\e}$ depending only on the constant $C_6$, such that for any $\e<\bar{\e}$ \eqref{eq:fin} is a contradiction. 
So we can conclude that $T^{\e,\delta}=T$ and so, by applying Lemma \ref{lem:e1} and the Lemma \ref{lem:e2}, $\ue$ and $\be$ are uniformly bounded in $C([0,T);H^{1}(\T))$, namely
\begin{equation}\label{eq:beb}
\sup_{t\in[0,T)}(\|\ue\|_{H^{1}}+\|\be\|_{H^{1}})\leq C.
\end{equation}
\\
{\em Step 3. Passage to the limit.}\\
\\
We are going to show that $(\ue,\be)$ converge to the unique global smooth solutions of \eqref{eq:MHD} with initial data \eqref{eq:ID}.
First, we note that since $T^{\e,\delta}=T$ there exists $(u^*,B^*)\in C([0,T);H^{1}(\T;\R^3)$ such that up to a subsequence the following convergences hold
\begin{equation}
\begin{aligned}
&\ue\rightarrow u^{*}\textrm{ weakly$^{*}$ in }C([0,T); H^{1}(\T)),\\
&\be\rightarrow B^{*}\textrm{ weakly$^{*}$ in }C([0,T); H^{1}(\T)).
\end{aligned}
\label{conv}
\end{equation}
Moreover, by Lemma \ref{lem:e1}, after integrating in time, and the global bound in $C([0,T);H^{1}(\T;\R^3)$ we have also 
\begin{equation}\label{eq:pb}
\e\int\|\pt\be\|_{2}^2\,\leq C.
\end{equation}
Finally, by using the Gagliardo-Nirenberg inequality, \eqref{eq:beb} and the bound on $j$ in Lemma \ref{lem:est1} we get easily that 
\begin{equation}\label{eq:eb}
\int\|\ee\|_{2}^{2}\,\leq C.
\end{equation}
Where the constants $C>0$ are independent on $\e$. 
We want to prove that $(u^{*}, B^{*})$ is a weak solution of the system \eqref{eq:MHD}. Let us multiply the first equations of \eqref{eq:main} by $\phi\in C_{c}([0,T);C^{\infty}_{per}(\T))$ with $\dive\phi=0$ and the second equations by $\psi\in C_{c}([0,T);C^{\infty}_{per}(\T))$. Specifically, from the equation for the velocity we get:    
\begin{equation*}
\begin{aligned}
\int\int-\ue\partial_t\phi+\nabla\ue\nabla\phi+((\ue\cdot\nabla)\ue\phi)+\int\ue_0\phi(x,0)&-\int\int\curl\be\times\be\phi\\
&=\int\int\e(\pt\ee\times\be)\phi,
\end{aligned}
\end{equation*}
and from the equation for the magnetic field:
\begin{equation*}
\begin{aligned}
&\int\!\!\int-\be\pt\psi+\nabla\be\nabla\psi-(\ue\times\be)\curl\psi+\int\be_0\psi(x,0)=
\int\!\!\int\e\ptt\be\psi.
\end{aligned}
\end{equation*}
By using \eqref{conv} and by using the equations \eqref{eq:mns} to get the necessary estimates in time we can easily pass to the limit in all the terms of the previous equalities except the terms on the right-hand sides. We want to prove that
\begin{equation}\label{eq:vi}
\begin{aligned}
&\e\int\int\pt\ee\times\be\phi\rightarrow 0\qquad\textrm{ as }\e\rightarrow 0,\\
&\e\int\int\ptt\be\psi\rightarrow 0\qquad\textrm{ as }\e\rightarrow 0.
\end{aligned}
\end{equation}
Let us start with the first term
\begin{align*}
&\e\int\int\pt\ee\times\be\phi=\e\int\int\pt(\ee\times\be)\phi-\e\int\int\ee\times\pt\be\phi\\
&=\e\int\ee_0\times\be_0\phi(x,0)-\e\int\int(\ee\times\be)\pt\phi-\e\int\int\ee\times\pt\be\phi.
\end{align*}
Then by using the  estimates \eqref{eq:pb}, \eqref{eq:eb}  and the uniform bounds on the initial data we get that 
\begin{equation*}
\e\left|\int\int\pt\ee\times\be\phi\right|\to 0\quad\textrm{ as }\quad\e\to 0.
\end{equation*}
Concerning the second one we have 
\begin{align*}
\e\int\int\ptt\be\phi&=-\e\int\int\pt\be\pt\phi+\e\int\pt\be_0\phi(x,0)\\
&=\e\int\int\be\ptt\phi-\e\int\be_0\pt\phi(x,0)\\
&+\e\int\pt\be_0\phi(x,0).
\end{align*}
Then, by using Lemma \ref{lem:est1} and the uniform bounds on the initial data
\begin{equation*}
\left|\e\int\int\ptt\be\phi\right| \to 0\quad\textrm{ as }\quad\e\to 0.
\end{equation*}
\\
{\em Step 4. Identification of the limit.}\\
\\
The final step of the proof is to prove that $(u^{*},B^{*})$ are the unique smooth solutions of \eqref{eq:MHD}-\eqref{eq:ID}. First we prove that $u^{*}$ and $B^{*}$ have vanishing third component because $u_0$ and $B_0$ are in $\R^{2}$. 
Let $\tilde u=(u^{*}_1,u^{*}_2)$ and $\tilde B=(B^{*}_1, B^{*}_2)$.  Since $u^*$ and $B^*$ do not depend on $x_3$ we have  that 
$\dive\tilde{u}=\dive\tilde{B}=0$ and the equations for $u^{*}_3$ and $B^{*}_3$ satisfied in the sense of distributions read as follows 
\begin{equation}\label{eq:3d}
\begin{aligned}
\pt\uu-\Delta\uu+\tilde{u}\cdot\nabla\uu-\tilde{B}\cdot\nabla\bb&=0\\
\pt\bb-\Delta\bb+\tilde{u}\cdot\nabla\bb-\tilde{B}\cdot\nabla\uu&=0
\end{aligned}
\end{equation}
Because of \eqref{conv} we can multiply the first equation by $\uu$ and the second by $\bb$. After integrating by parts and adding up we get 
\begin{equation}\label{eq:3d1}
\frac{d}{dt}\left(\int|\uu|^2+|\bb|^2 \right)+2\int|\nabla\uu|^2+2\int|\nabla\bb|^2=0,
\end{equation}
by Gronwall lemma we have that $\uu$ and $\bb$ vanish. Then, $(u^{*},B^{*})$ are a weak solutions of the Cauchy problem \eqref{eq:MHD}-\eqref{eq:ID}. By using the uniqueness result of the Theorem \ref{teo:exMHD} we get that 
$(u^{*},B^{*})=(u,B)$.


\begin{thebibliography}{10}

\bibitem{AIM} 
D.~Ars\'enio, S.~Ibrahim, N.~Masmoudi, 
\emph{A derivation of the Magnetohydrodynamic system from the Navier-Stokes-Maxwell systems}, ARMA,
\textbf{216} (3), (2015),  796--812.

\bibitem{BNP04}
Y.~Brenier, R.~Natalini, and M.~Puel, 
\emph{On a relaxation approximation of the incompressible {N}avier--{S}tokes equations}, Proc. Amer. Math. Soc.
\textbf{132} (2004), 1021--1028. 

\bibitem{BG}     
H.~Brezis and T.~Gallouet, \emph{Nonlinear Schr\"odinger evolution equations}, Nonlinear Anal., \textbf{4}, (1980), no. 4, 677-681.


\bibitem{D}
P.A.~Davidson, \emph{An Introduction to Magnetohydrodynamics}, Cambridge University Press, 2001.



\bibitem{DM10}
M.~Di Francesco, D.~Donatelli, 
\emph{Singular convergence of nonlinear hyperbolic chemotaxis systems to {K}eller-{S}egel type models},
Discrete Contin. Dyn. Syst. Ser. B, \textbf{13}, (2010), no. 1, 79--100.


\bibitem{DM02}
D.~Donatelli and P.~Marcati, 
\emph{Singular limits for nonlinear hyperbolic systems}, Evolution equations, semigroups and functional analysis (Milano,
  2000), Progr. Nonlinear Differential Equations Appl., vol.~\textbf{50}, Birkh\"auser,
  Basel, 2002, 79--96.

\bibitem{DM04}
D.~Donatelli and P.~Marcati, \emph{Convergence of singular limits
for multi-{D} semilinear hyperbolic systems to parabolic systems}, Trans. Amer. Math. Soc.,
  \textbf{356} (2004), 2093--2121 (electronic).






\bibitem{DL09}
D.~Donatelli, C.~Lattanzio, 
\emph{On the diffusive stress relaxation for multidimensional viscoelasticity},
Commun. Pure Appl. Anal., \textbf{8}, (2009), no. 2, 645--654.

\bibitem{EM1990}
A.~C.~Eringen and G.~A.~Maugin, {\em Electrodynamics of Continua: II. Fluids and Complex Media}, 
Springer, New York, 1990.


\bibitem{GIM}
P.~Germain, S.~Ibrahim and N.~Masmoudi,
\emph{Well-posedness of the {N}avier-{S}tokes-{M}axwell equations}, Proc. Roy. Soc. Edinburgh Sect. A., \textbf{144}, (2014), no. 1, 71--86.

\bibitem{JM12}
J.~Jang and N.~Masmoudi,
\emph{Derivation of {O}hm's law from the kinetic equations}, SIAM J. Math. Anal.,
\textbf{44}, (2012), no. 5, 3649--3669.


\bibitem{JL2012}
S.~Jiang and F.~Li, {\em Rigorous derivation of the compressible magnetohydrodynamic equations
from the electromagnetic fluid system}, Nonlinearity, {\bf 25},  (2012), no. 6, 1735--1752.

\bibitem{KS1986i}
S.~Kawashima and Y.~Shizuta, {\em Magnetohydrodynamic approximation of the complete equations
for an electromagnetic fluid}, Tsukuba J. Math., {\bf 10},  (1986), 131--49.

\bibitem{KS1986ii}
S.~Kawashima and Y.~Shizuta, {\em Magnetohydrodynamic approximation of the complete equations
for an electromagnetic fluid: II}, Proc. Japan Acad. Ser. A Math. Sci., {\bf 62},  (1986), 181--184.



\bibitem{MR00}
P.~Marcati and B.~Rubino, \emph{Hyperbolic to parabolic relaxation
theory for quasilinear first order systems}, J. Differential Equations, \textbf{162}
  (2000), 359--399.

\bibitem{M07}
N.~Masmoudi, \emph{Remarks about the inviscid limit of the
    {N}avier-{S}tokes system}, Comm. Math. Phys. \textbf{270}, (2007),
  777--788.

\bibitem {M10}
N.~Masmoudi, 
\emph{Global well posedness for the {M}axwell-{N}avier-{S}tokes system in 2{D}}, J. Math. Pures Appl.,
\textbf{93}, (2010), no. 6, 559--571.
      
\bibitem{ST}
M.~Sermange and R.~Temam, \emph{Some mathematical questions related to the {M}{H}{D} equations}, Comm. Pure Appl. Math., \textbf{36}, (1983), no. 5,  635-664.




\end{thebibliography}
\end{document}